\documentclass[11pt]{article}
\usepackage{amssymb}
\newtheorem{precor}{{\bf Corollary}}

\newtheorem{precon}{{\bf Conjecture}}

\newtheorem{prealphcon}{{\bf Conjecture}}

\newtheorem{predefin}{{\bf Definition}}

\newtheorem{preexm}{{\bf Example}}

\newtheorem{preappl}{{\bf Application}}

\newtheorem{prelem}{{\bf Lemma}}

\newtheorem{preproof}{{\bf Proof.\ }}

\newenvironment{proof}[1]{\begin{preproof}{\rm
               #1}\hfill{$\blacksquare$}}{\end{preproof}}
\newtheorem{pretheorem}{{\bf Theorem}}

\newenvironment{theorem}{\begin{pretheorem}{\hspace{-0.5
               em}{\bf.\ }}}{\end{pretheorem}}
\newtheorem{prealphtheorem}{{\bf Theorem}}

\newtheorem{prealphlem}{{\bf Lemma}}

\newtheorem{prepro}{{\bf Proposition}}

\newtheorem{preprb}{{\bf Problem}}

\newtheorem{prerem}{{\bf Remark}}

\newtheorem{preapp}{{\bf Application}}

\newtheorem{prequ}{{\bf Question}}

\def\conct[#1,#2]{\mbox {${#1} \leftrightarrow {#2}$}}
\def\dconct[#1,#2]{\mbox {${#1} \rightarrow {#2}$}}
\def\deg[#1,#2]{\mbox {$d_{_{#1}}(#2)$}}
\def\mindeg[#1]{\mbox {$\delta_{_{#1}}$}}
\def\maxdeg[#1]{\mbox {$\Delta_{_{#1}}$}}
\def\outdeg[#1,#2]{\mbox {$d_{_{#1}}^{^+}(#2)$}}
\def\minoutdeg[#1]{\mbox {$\delta_{_{#1}}^{^+}$}}
\def\maxoutdeg[#1]{\mbox {$\Delta_{_{#1}}^{^+}$}}
\def\indeg[#1,#2]{\mbox {$d_{_{#1}}^{^-}(#2)$}}
\def\minindeg[#1]{\mbox {$\delta_{_{#1}}^{^-}$}}
\def\maxindeg[#1]{\mbox {$\Delta_{_{#1}}^{^-}$}}

\def\dre[#1,#2,#3]{\mbox {${\cal E}^{^{#3}}(#1,#2)$}}
\def\var[#1,#2]{\mbox {${\rm Var}_{_{#1}}(#2)$}}
\def\ls[#1]{\mbox {$\xi^{^{#1}}$}}
\def\hom[#1,#2]{\mbox {${\rm Hom}({#1},{#2})$}}
\def\onvhom[#1,#2]{\mbox {${\rm Hom^{v}}(#1,#2)$}}
\def\onehom[#1,#2]{\mbox {${\rm Hom^{e}}(#1,#2)$}}
\def\core[#1]{\mbox {$#1^{^{\bullet}}$}}
\def\cay[#1,#2]{\mbox {${\rm Cay}({#1},{#2})$}}
\def\sch[#1,#2,#3]{\mbox {${\rm Sch}({#1},{#2},{#3})$}}
\def\cays[#1,#2]{\mbox {${\rm Cay_{s}}({#1},{#2})$}}
\def\dirc[#1]{\mbox {$\stackrel{\rightarrow}{C}_{_{#1}}$}}
\def\cycl[#1]{\mbox {${\bf Z}_{_{#1}}$}}

\begin{document}
%\setcounter{page}{183}
%{\footnotesize AAA {\bf ?} (200?) ?--?}\\
%\maketitle

\begin{center} 
{\Large \bf A note on b-coloring of Kneser graphs}\\
\vspace{0.3 cm}
{\bf Saeed Shaebani}\\
{\it School of Mathematics and Computer Science}\\
{\it Damghan University}\\
{\it P.O. Box {\rm 36716-41167}, Damghan, Iran}\\
{\tt shaebani@du.ac.ir}\\ \ \\
\end{center}
\begin{abstract}
\noindent In this short note, the purpose is to provide an upper bound
for the b-chromatic number of Kneser graphs. Our bound improves
the upper bound that was presented by Balakrishnan and Kavaskar in
[b-coloring of Kneser graphs, {\em Discrete Appl. Math.}\ {\bf
160} (2012), 9-14].\\


\end{abstract}
%%%%%%%%%%%%%%%%%%%%%%%%%%%%%%%%%%%%%%%%%%%%%%%%%%%%%%%%%%%%%%%%%%%%%%%
%%%%%%%%%%%%%%%%%%%%%%%%%%%%%%%%%%%%%%%%%%%%%%%%%%%%%%%%%%%%%%%%%%%%%%%%
\section{Introduction}

In this note, simple graphs whose vertex sets are
nonempty and finite are considered. Let $G$ be a graph with vertex 
set $V(G)$. A {\it coloring} of $G$ stands for a function $f:V(G)\rightarrow C$ such
that for each $c$ in $C$, the set $f^{-1}(c)$ is independent; in
this case, we think of each $c$ in $C$ as a {\it color} and call
$f^{-1}(c)$ a {\it color class} of $f$.

Let $G$ be a graph and $f:V(G)\rightarrow C$ be a coloring of $G$. 
The vertex $v$ of $G$ is said to be a {\it color-dominating} vertex with
respect to $f$ if $f(N[v])=C$, i.e., the vertex $v$ has all 
colors in its closed neighborhood. Also, the coloring $f:V(G)\rightarrow C$
is called a {\it b-coloring} of $G$ whenever
each of its color classes contains at least one color-dominating vertex.
The {\it b-chromatic number} of $G$, denoted by $\varphi (G)$, is defined to 
be the maximum positive integer $k$ for which $G$ admits a b-coloring 
$f:V(G)\rightarrow C$ with $|C|=k$. This concept was introduced by Irving and 
Manlove in $1999$ in \cite{irv}, and since then there exists an extensive literature 
on it; see \cite{survey} for a survey. 

Suppose that $n$ and $m$ are positive integers and $n\geq m$. 
The {\it Kneser graph} $KG(n,m)$ is the graph whose vertex set 
is the set of all $m$-subsets of $\{1,2,\ldots ,n\}$, in which two 
vertices $A$ and $B$ are declared to be adjacent iff $A \cap B = \varnothing$.
In \cite{bala,haji,jav,sha}, b-coloring of Kneser graphs has been investigated.

Every $d$-regular graph $G$ satisfies $\varphi (G) \leq d+1$ \cite{irv}.
Kratochv\'{\i}l, Tuza, and Voigt \cite{kratv} showed that for each $d$
there are only finitely many $d$-regular graphs up to
isomorphism whose b-chromatic numbers are less than or equal to $d$. 
So, finding such regular graphs is of interest. In this regard, Balakrishnan and Kavaskar
\cite{bala} presented some desired Kneser graphs meeting this property.
\begin{theorem}{ \cite{bala} \label{bala} Let $n \geq 2$ and $i \geq 0$.
Also, let $d$ be the degree-regularity of the Kneser graph $G=KG(2n+k,n)$. If
$|V(G)| \leq 2d+2-2i$, then $\varphi (G) \leq d-i$.
}
\end{theorem}
The aim of this short note is to provide an
improvement of Theorem \ref{bala}.

\section{The main result}

This section concerns the main result of the note; as follows.\\
In Theorem \ref{bala}, the statement $|V(G)| \leq 2d+2-2i$ 
is equivalent to $\left\lceil \frac{|V(G)|-2}{2} \right\rceil  \leq d-i$. 
Therefore, the upper bound in this Theorem, which is $d-i$, is greater than or equal to 
$\left\lceil \frac{|V(G)|-2}{2} \right\rceil$. In the next theorem, we provide
a sharp upper bound for the b-chromatic number of Kneser graphs, which is asymptotic to $\frac{|V(G)|}{3}$.

\begin{theorem}{ \label{Themain} For fixed $n\geq 2$, the Kneser graph $G_{k}:=KG(2n+k,n)$ satisfies
\begin{center}{
$\varphi (G_{k}) \leq (1+o(1))\frac{|V(G_{k})|}{3},$
}\end{center}
where the $o(1)$ term tends to zero as $k$ tends to infinity.
}
\end{theorem}

\begin{proof}{
Let $\mathcal{C}$ be the set of color classes of an arbitrary b-coloring
of $G_{k}$. For each color class $S$ in $\mathcal{C}$, we set 
$S^{\circ}:=\bigcap_{A\in S} A$; and call $S$ a {\it non-intersecting} color class
whenever $S^{\circ} = \varnothing$. Let us denote by $\mathcal{I}$ the set
$\{S\in \mathcal{C}\ |\ S^{\circ} \neq \varnothing\}$.

Consider two distinct color classes $S$ and $T$ in $\mathcal{C}$; and let 
$\hat{S}$ be a color-dominating vertex of $S$. The vertex $\hat{S}$ is 
adjacent to a vertex of $T$, say $T_{1}$. So, $\hat{S} \cap T_{1} = \varnothing$.
Since $S^{\circ} \cap T^{\circ} \subseteq \hat{S} \cap T_{1}$, we have 
$S^{\circ} \cap T^{\circ} = \varnothing$. This shows that the function 
$f:\mathcal{I} \rightarrow \{1,2,\ldots ,2n+k\}$ that assigns the minimum of $S^{\circ}$ to every
$S$ in $\mathcal{I}$, is an injective mapping. Therefore, $|\mathcal{I}| \leq 2n+k$.

Each non-intersecting color class of $\mathcal{C}$ contains at least three vertices of
$G_{k}$. Hence, $|\mathcal{C}|-|\mathcal{I}| \leq \frac{|V(G_{k})|-|\mathcal{I}|}{3}$.
Accordingly, $|\mathcal{C}| \leq \frac{|V(G_{k})|+2|\mathcal{I}|}{3} \leq 
\frac{|V(G_{k})|+2(2n+k)}{3}$. Now, since
$\lim_{k \to \infty} \frac{2(2n+k)}{{2n+k \choose n}}=0$, we conclude that
$\varphi (G_{k}) \leq (1+o(1))\frac{|V(G_{k})|}{3}$, which is desired.
}
\end{proof}

In view of the proof of Theorem \ref{Themain},
we proved that for any fixed positive integer $n$, the b-chromatic number of Kneser graph
$G_{k}:=KG(2n+k,n)$ is less than or equal to $\mathcal{U}(G_{k}):=\frac{|V(G_{k})|+2(2n+k)}{3}$.
The upper bound $\mathcal{U}(G_{k})$ is sharp for $n=1$.

Let us regard an arbitrary integer $n\geq 2$ as fixed.
Asymptotically in $k$, the bound $\mathcal{U}(G_{k})$ is $\frac{|V(G_{k})|}{3}$.
In \cite{jav}, Javadi and Omoomi showed 
that for $n=2$, the b-chromatic number of $G_{k}$ is asymptotic to $\frac{|V(G_{k})|}{3}$.
Hence, for $n=2$, this bound is asymptotically correct; i.e., the ratio
$\frac{\varphi (G_{k})}{\mathcal{U}(G_{k})}$
goes to $1$ as $k$ tends to infinity.

Since $d+1$ is an upper bound for the b-chromatic number of any $d$-regular graph,
it is worth pointing out that for a fixed positive 
integer $n\geq 2$, if $d_{k}$ denotes the degree of any vertex of $G_{k}$,
then the upper bound in Theorem \ref{Themain} is asymptotically $\frac{d_{k}}{3}$, because the ratio 
$\frac{d_{k}}{|V(G_{k})|}$ tends to $1$ as $k$ tends to infinity.

\end{document}